\documentclass[12pt]{article}
\usepackage{graphicx}
\usepackage{color,latexsym,amsfonts,amssymb}
\usepackage{bm,bbm}     
\usepackage{amsthm} 
\usepackage{amsmath}
\usepackage{subfigure}
\usepackage{booktabs} 
\usepackage{boxedminipage}  
\usepackage{natbib} 
\usepackage{algorithm,algpseudocode}  

\DeclareMathOperator*{\argmin}{argmin}

\newtheorem{theorem}{Theorem}

\newtheorem{lemma}{Lemma}

\newtheorem{assumption}{Assumption}

\newcommand{\Var}{\mathrm{Var}}

\title{\vspace{-2cm} Team Variance Optimization of $n$-Player Stochastic Games with Separately Controlled Chains}
\author{Li Xia\thanks{Li Xia is with the School of Business and the Guangdong Province Key Laboratory of Computational Science, Sun Yat-Sen University, Guangzhou 510275, China. (email: xiali5@sysu.edu.cn).}}
\date{}
\begin{document}
\maketitle

\begin{abstract}
In this paper, we study a subclass of $n$-player stochastic games,
in which each player has their own internal state controlled only by
their own action and their objective is a common goal called team
variance which measures the total variation of the random rewards of
all players. It is assumed that players cannot observe each others'
state/action. Thus, players' internal chains are controlled
separately by their own action and they are coupled through the
objective of team variance. Since the variance metric is not
additive or Markovian, the dynamic programming principle fails in
this problem. We study this problem from the viewpoint of
sensitivity-based optimization. A difference formula and a
derivative formula for team variance with respect to policy
perturbations are derived, which provide sensitivity information to
guide decentralized optimization. The existence of a stationary pure
Nash equilibrium policy is derived. We further develop a bilevel
optimization algorithm that iteratively updates the team mean at the
outer level and minimizes the team variance at the inner level,
where the team mean serves as a signal to coordinate the
optimization of $n$ players in a decentralized manner. We prove that
the algorithm can converge to a strictly local minimum or a
first-order stationary point in the space of mixed policies.
Finally, we demonstrate the effectiveness of our approach using a
numerical experiment of energy management in smart grid, where the
assumption of separately controlled chains holds naturally.

%

\end{abstract}

\textbf{Keywords}: Markov decision process, stochastic game, team
variance, sensitivity-based optimization, decentralized optimization

\section{Introduction}\label{section_intro}

Variance is a widely adopted metric to measure the deviation of a
random variable. It can be used as an optimization objective when
the stability, safety, or risk of systems is concerned. For Markov
decision processes (MDPs) or reinforcement learning, the system
rewards are a series of random variables controlled by policies and
the classical optimization objective is the expectation of these
random rewards. However, it is reasonable to consider other metrics
of random rewards such that the policy optimization is superior,
which belongs to the framework of risk-sensitive MDPs or even
distributional reinforcement learning \citep{Bellemare23}. Such
framework of risk-sensitive decision-making is particularly useful
for applications where uncertainty and variability are critical,
such as finance, supply chain management, and robotics. Some widely
adopted risk measures in risk-sensitive MDPs include variance
\citep{Sobel82,Xia20}, exponential utility function
\citep{Howard1972,Borkar02,Cavazos-Cadena23,Blancas-Rivera24},
value-at-risk (VaR) \citep{Li2022}, conditional value-at-risk (CVaR)
\citep{Bauerle2011,Xia2023}, probability \citep{White93,Wen24},
prospect theory \citep{Etesami18,Wu24}, etc. When variance is
adopted as a risk metric, it is referred to as the mean-variance
optimization in MDPs, which suffers from the failure of dynamic
programming principle caused by the non-additivity or non-separable
property of variance metrics.

Some pioneering works about variance-related MDPs refer to
\cite{Sobel82} and \cite{Filar89}. There are a series of relevant
works along this research stream where the variance-constrained or
variance-penalized mean metrics of MDPs are studied with the aid of
mathematical programming \citep{White92,Chung94,Sobel94,Borkar24}.
Another research stream is directly studying the variance-related
criterion of MDPs in a restricted policy set where the mean metric
already reaches its optimum value
\citep{Hernandez99,Guo09b,Huang18,Xia18a}. Reinforcement learning is
also a research stream to study the variance-related risk measure in
stochastic dynamic systems \citep{Tamar12,Prashanth13,Bisi20}, where
policy gradient-based algorithms are usually employed to do
optimization in a data-driven manner. Furthermore, in financial
engineering, there also exists a stream of research efforts that
study the mean-variance optimization for multi-period portfolio
selection problems \citep{Li00,Zhou00,Zhou04}, where an embedding
method is proposed to particularly solve this portfolio model in the
framework of stochastic control. Nevertheless, the problem of
variance optimization of stochastic dynamic systems is still
challenging, especially on the issues of algorithmic study and
efficient computation of optimal policies.

The aforementioned literature focuses on the optimization of
variance-related risk metrics in the scenario of a single
decision-maker. It is a natural extension to consider risk metrics
in the scenario of multiple decision-makers, which is called
risk-sensitive stochastic games in mathematics or multi-agent
reinforcement learning in computer science. However, stochastic
games and multi-agent reinforcement learning have significant
challenges caused by the coupling behaviors of joint actions of
multiple decision-makers, which makes the decision-making
environment not stationary for each individual decision-maker. Local
information observation is another significant challenge, since it
widely exists that each decision-maker only observes its own local
information and has no willingness or communication capability to
share their information.

Constrained by these challenges, most of the existing literature in
stochastic games and multi-agent reinforcement learning focuses on
the risk-neutral objectives, i.e., the expectation of random
rewards. The studies of multi-agent reinforcement learning
concentrate on the sample efficiency and learning capability to find
$\epsilon$-approximate Nash equilibrium (NE) policies, and can be
categorized into cooperative algorithms (such as team games or
potential games) and noncooperative algorithms (such as zero-sum or
general-sum games). Since the remarkable success of AlphaGo, a
substantial and growing body of literature on multi-agent
reinforcement learning has emerged. \cite{Yongacoglu21} present a
two-timescale Q-learning algorithm that achieves team optimality in
a stochastic team game under full state observability but no action
sharing, and the algorithm's provable convergence is also provided.
\cite{Yan24} provably find lower confidence bounds and sample
complexity results for a model-based reinforcement learning
algorithm that can find an $\epsilon$-approximate NE of a two-player
zero-sum Markov game. For a comprehensive understanding on the
theories and algorithms of multi-agent reinforcement learning,
audience can refer to a recent survey by \cite{Zhang21}. For
stochastic games, more attention has been paid to the analysis of
mathematical models and the existence of equilibrium policies.
Considering the aforementioned challenges in stochastic games, some
special structures are usually assumed to leverage the rigorous
analysis and proofs. For example, \cite{Etesami24} aims to find
stationary NE policies in a local information observable $n$-player
stochastic game, with assumption that players' internal chains are
driven by independent transition probabilities. The author develops
polynomial-time learning algorithms based on dual averaging and dual
mirror descent, and demonstrates the effectiveness of finding an
$\epsilon$-NE policy using numerical experiments for energy
management in smart grid.

In contrast to the risk-neutral objectives, there exists much less
literature studying the risk-sensitive stochastic games with
variance-related risk measures, where players seek to maximize
expected rewards (mean) while simultaneously minimizing variability
(variance) in outcomes. The fundamental difficulties include the
non-stationarity and the local information observability caused by
multiple players, and the failure of dynamic programming caused by
variance-related metrics. \cite{Parilina21} study a cooperative
stochastic game with mean-variance payoff functions and define
characteristic function using a maxmin approach to determine
imputations. \cite{Escobedo-Trujillo24} study the variance
optimality in a two-player stochastic differential game and derive
sufficient conditions for the existence of relaxed strategies that
optimize the limiting variance of a constrained performance index of
each player in the continuous-time framework. 
There are some other works studying the mean-variance criterion in
the model of stochastic differential game for insurance problems
\citep{Li22}. However, these aforementioned works always assume the
full state observability and the algorithmic solutions for finding
NE policies are absent. To the best of our knowledge, there does not
exist literature work that studies the variance optimization
criterion in $n$-player stochastic games with local information
observability.

In this paper, we focus on the team variance optimization of
$n$-player stochastic games with separately controlled chains, where
the metric of team variance is defined to measure the deviation of
$n$ players' random rewards in steady state. The concept of team
variance is similar to to the group variance that is adopted in a
statistical tool called analysis of variance
\citep{Fisher25,Sawyer09}. The team variance can be partitioned into
the squared deviations of the rewards from their respective player
means (within-player variance) plus the squared deviations of the
player means from the overall mean (between-player variance), which
reflect the stability or risk within players and the evenness or
fairness between players, respectively. Our objective is to minimize
the team variance of the whole system in a decentralized manner,
where each player makes decision only based on their own local
information (individual state, action, and reward). We assume that
the internal state of each player is separately controlled by their
own action, not affected by other players' actions. This assumption
is reasonable in many practical applications, such as the energy
management problem in our experiment, and similar assumption is also
adopted in other literature work \citep{Etesami24}. Since dynamic
programming fails for the non-additivity of variance metrics, we
study this problem from a new perspective called sensitivity-based
optimization \citep{Cao07}, which focuses on the performance
sensitivity when the system policy or parameters have perturbations.
For this stochastic game, we derive the so-called team variance
difference formula and derivative formula, which quantify the change
of team variances when players' policies have changes. Based on
these formulas, we prove that the minimum variance can be attained
by stationary deterministic policies, i.e., a stationary pure Nash
equilibrium always exists for this stochastic game. We further
convert this problem to a bilevel optimization problem where the
inner level is $n$ independent optimization problems executed by
each player and the outer level is the optimization of team mean
which can be viewed as a signal to coordinate the inner optimization
of $n$ players. We develop a decentralized optimization algorithm
and prove its convergence to a strictly local minimum or a
first-order stationary point in the space of mixed policies.
Finally, we use a numerical experiment of energy management in smart
grid to demonstrate the effectiveness of our approach.

The main contributions of this paper are twofold. First, we propose
a novel approach for studying team variance optimization in
$n$-player stochastic games with local information observability. To
the best of our knowledge, this is the first work to address the
dual challenges posed by the variance metric and local information
constraints in stochastic games. A key strength of our approach lies
in the derivation of sensitivity optimization formulas for team
variances, which provides a promising framework for handling such
problems. Second, we develop an effective algorithm to optimize team
variance of stochastic games in a decentralized manner, and we
rigorously prove its convergence. While most existing literature on
stochastic games primarily focuses on the mathematical analysis of
equilibrium existence, algorithmic studies remain scarce despite
their importance for practical applications. Our work makes a
complementary contribution to the algorithmic results of stochastic
games, offering both theoretical insights and practical tools for
variance-aware optimization in multi-agent systems.

The remainder of this paper is organized as follows. In
Section~\ref{section_model}, we introduce the mathematical model of
team variance optimization in stochastic games. In
Section~\ref{section_result}, we present our main results, including
the sensitivity optimization formulas and the development of our
algorithm. Section~\ref{section_experiment} demonstrates the
effectiveness of our approach by using a numerical example of energy
management. Finally, we conclude this paper in
Section~\ref{section_conclusion}.

\section{Problem Formulation}\label{section_model}
We consider a discrete-time stochastic game with $\mathcal
I:=\{1,2,\dots,n\}$ players, where each player $i \in \mathcal I$
has its own finite state space $\mathcal S_i$ and finite action
space $\mathcal A_i$. The joint space of state and action of the
whole stochastic game is denoted by $\mathcal S = \times_{i=1}^{n}
\mathcal S_i$ and $\mathcal A = \times_{i=1}^{n} \mathcal A_i$,
respectively, where $\times$ indicates the Cartesian product. At
time epoch $t=0,1,\dots$, the system state is $\bm s_t=(s_{1,t},
s_{2,t}, \dots, s_{n,t})$ and the action is $\bm a_t=(a_{1,t},
a_{2,t}, \dots, a_{n,t})$, where $s_{i,t} \in \mathcal S_i$,
$a_{i,t} \in \mathcal A_i$, and $i \in \mathcal I$. We denote
$r_i(\bm s_t,\bm a_t)$ as the (random) reward received by player $i$
at time $t$. The system will transit from the current state $\bm
s_t$ to the next state $\bm s_{t+1}$ according to a transition
probability $p(\bm s_{t+1}|\bm s_t, \bm a_t)$, where action $\bm a_t
\in \mathcal A$ is adopted at time $t$ and $\bm s_t, \bm s_{t+1} \in
\mathcal S$. We have $p(\bm s_{t+1}|\bm s_t, \bm a_t) =
\prod_{i=1}^{n}p_i(s_{i,t+1}|s_{i,t}, \bm a_t)$, where
$p_i(s_{i,t+1}|s_{i,t}, \bm a_t)$ is the transition probability of
internal state of player $i$. In this paper, we consider a subclass
of stochastic games that satisfy the following assumption.

\begin{assumption}\textbf{(Separately controlled chains)}
\label{assm_1} The state $s_i$ and reward $r_i$ of each player~$i
\in \mathcal I$ is controlled only by its own action $a_i$, i.e.,
$p_i(s'_i|s_i, \bm a) = p_i(s'_i|s_i,a_i)$ and $r_i(\bm s, \bm
a)=r_i(s_i,a_i)$, where $s_i, s'_i \in \mathcal S_i$, $a_i \in
\mathcal A_i$, $\bm s \in \mathcal S$, and $\bm a \in \mathcal A$.
\end{assumption}

Assumption~\ref{assm_1} indicates that the dynamics and reward of
player $i$ are affected only by its own action, and independent of
other players' actions. This assumption has reasonable scenarios in
practice, such as the example of energy management discussed later
in Section~\ref{section_experiment}, where each microgrid (player)
only controls its own state (energy storage level and renewable
generator power) and does not affect other microgrids' state.
Moreover, in this paper we also assume that each player has local
information observable, as stated by the following assumption.

\begin{assumption}\textbf{(Local observations)}
\label{assm_2} Each player~$i \in \mathcal I$ only observes its own
information $(s_i, a_i, r_i)$, and has no access to other players'
information.
\end{assumption}

Assumption~\ref{assm_2} is also reasonable in many applications,
including the example of energy management in
Section~\ref{section_experiment}, since players are usually selfish
and have no willingness to share their own information to others.
Therefore, each player $i$ at time $t$ has its local information
$h_{i,t}$, which belongs to the history set $\mathcal H_{i,t} :=
\{(s_{i,l}, a_{i,l}, r_{i,l}) : l=0,1,\dots,t-1 \} \bigcup
\{s_{i,t}\}$. We denote the policy of player $i$ as $u_i :=
(u_{i,0}, u_{i,1}, \dots, u_{i,t},\dots)$, where its element
$u_{i,t}$ is called the decision rule at time $t$ and it is a
mapping from history information to a probability distribution on
action space, i.e., $u_{i,t} : \mathcal H_{i,t} \rightarrow \mathcal
P(\mathcal A_i)$, where $\mathcal P(\cdot)$ indicates the space of
probability distributions on a set. We call $u_i$
\emph{history-dependent randomized policy} and its policy space is
denoted by $\mathcal U^{\rm{HR}}_i$. An element of $u_{i,t}$, say
$u_{i,t}(a_i|h_{i,t})$, indicates the probability of adopting action
$a_i$ when local history information $h_{i,t}$ is observed, where
$a_i \in \mathcal A_i$ and $h_{i,t} \in \mathcal H_{i,t}$. If
$u_{i,t}(a_i|h_{i,t}) = u_{i,t}(a_i|s_{i,t})$ always holds, we call
$u_i$ \emph{Markovian randomized policy} and its policy space is
denoted by $\mathcal U^{\rm{MR}}_i$. If $u_{i,t}(\cdot|s_{i,t})$ is
a dirac distribution on $\mathcal A_i$, we call $u_i$
\emph{Markovian deterministic policy} and its policy space is
denoted by $\mathcal U^{\rm{MD}}_i$, and we may also use
$u_{i,t}(s_{i,t}) \in \mathcal A_i$ to represent the deterministic
action $a_{i,t}$. Furthermore, if $u_{i,t}(\cdot|\cdot)$ remains the
same under different $t$, we call it \emph{stationary randomized
policy} and its policy space is denoted by $\mathcal U^{\rm{SR}}_i$,
and $u_{i,t}(\cdot|\cdot)$ is degenerately denoted as
$u_{i}(\cdot|\cdot)$ for notational simplicity. If
$u_{i}(\cdot|\cdot)$ is further a dirac distribution, we call $u_i$
\emph{stationary deterministic policy} and its policy space is
denoted by $\mathcal U^{\rm{SD}}_i$.

The joint policy of the stochastic game is denoted by
$u=(u_1,u_2,\dots,u_n)$ and the joint policy space is the Cartesian
product of all players' policy spaces, such as $\mathcal U^{\rm HR}
:= \times_{i=1}^{n} \mathcal U^{\rm HR}_i$ and $\mathcal U^{\rm SD}
:= \times_{i=1}^{n} \mathcal U^{\rm SD}_i$, etc. In this paper, we
limit our attention to ergodic policies which ensure the Markov
chain associated with policy $u \in \mathcal U^{\rm HR}$ ergodic.
Since risk management is very important and variance is a popular
risk metric, we consider an optimization objective called \emph{team
variance} which measures the variance of all players' random
rewards, i.e.,
\begin{equation}\label{eq_teamvar}
J^u_{\sigma} := \mathbbm E^u \left\{ \sum_{i=1}^{n} \lim\limits_{T
\rightarrow \infty} \frac{1}{T}\sum_{t=0}^{T-1}[r_i(s_{i,t},
a_{i,t}) - \mu^u]^2 \right\},
\end{equation}
where $\mathbbm E^u\{\cdot\}$ indicates the expectation induced by
policy $u \in \mathcal U^{\rm HR}$ and $\mu^u$ is the long-run
average of all players' rewards defined as
\begin{equation}\label{eq_teammean}
\mu^u := \mathbbm E^u \left \{ \frac{1}{n}\sum_{i=1}^{n}
\lim\limits_{T \rightarrow \infty} \frac{1}{T}\sum_{t=0}^{T-1}
r_i(s_{i,t}, a_{i,t}) \right \}.
\end{equation}
Furthermore, for each player $i \in \mathcal I$, we can also define
the long-run mean and variance of this player's random rewards under
policy $u_i \in \mathcal U^{\rm HR}_i$
\begin{equation*}
\mu^{u_i}_i := \mathbbm E^{u_i} \left \{ \lim\limits_{T \rightarrow
\infty} \frac{1}{T}\sum_{t=0}^{T-1} r_i(s_{i,t}, a_{i,t}) \right \},
\end{equation*}

\begin{equation*}
J^{u_i}_{\sigma,i} := \mathbbm E^{u_i} \left \{ \lim\limits_{T
\rightarrow \infty} \frac{1}{T}\sum_{t=0}^{T-1} [r_i(s_{i,t},
a_{i,t}) - \mu^{u_i}_i]^2 \right \}.
\end{equation*}
Obviously, we have
\begin{equation*}
\mu^{u} = \frac{1}{n}\sum_{i=1}^{n}\mu^{u_i}_i.
\end{equation*}
With Assumption~\ref{assm_1} of separately controlled chains, we can
derive the following relation between the team variance of the game
and the variances of each player
\begin{equation}\label{eq_var-relation}
J^u_{\sigma} = \sum_{i=1}^n J^{u_i}_{\sigma,i} + \sum_{i=1}^n
[\mu^{u_i}_i- \mu^{u}]^2.
\end{equation}
We observe that the team variance $J^u_{\sigma}$ can be partitioned
into the squared deviations of the rewards from their respective
player's mean (within-player variance) plus the squared deviations
of the players' means from the overall mean (between-player
variance), which correspond to the first and second part of the
right-hand-side of (\ref{eq_var-relation}), respectively. The
relation in (\ref{eq_var-relation}) can be explained by the
\emph{law of total variance}, $\Var[X] = \mathbbm E[\Var[X|Y]] +
\Var[\mathbbm E[X|Y]]$. This concept is also similar to the
with-group and between-group variances which are widely used in
analysis of variance \citep{Fisher25,Sawyer09}. These two partitions
have physical or economic meanings in practice, which reflect the
stability or risk within players and the evenness or fairness
between players, respectively. Therefore, it is of significance to
optimize the team variance in stochastic games. In this paper, we
aim to find an optimal policy such that the team variance of the
stochastic game is minimized, i.e.,
\begin{flalign}\label{eq_prob0}
&\mathcal M: \hspace{6cm}  J^*_{\sigma} = \min_{u \in \mathcal
U^{\rm HR}} J^{u}_{\sigma} &
\end{flalign}
and $u^* \in \mathcal U^{\rm HR}$ is called an optimal policy if it
attains $J^{u^*}_{\sigma} = J^*_{\sigma}$.

Although Assumption~\ref{assm_1} ensures the dynamics of players
evolve separately, the objective of team variance couples the
decision-making processes of players. From the definition
(\ref{eq_teamvar}) and (\ref{eq_var-relation}), we can see that team
variance $J^u_{\sigma}$ is influenced by the value of team mean
$\mu^u$, which is determined by the collective behavior of all
players. Consequently, the decision-making processes of all players
are interconnected, and the optimization problem (\ref{eq_prob0})
constitutes a stochastic game. However, there are no existing
approaches to optimize the team variance in stochastic games, as far
as the authors are aware. The primary challenge stems from the
non-additivity and non-Markovian nature of variance metrics, which
renders the major methodology of dynamic programming inapplicable
\citep{Sobel82,Xia20}. Moreover, as stated in
Assumption~\ref{assm_2}, each player only has access to its local
information. Therefore, we have to solve the optimization problem
(\ref{eq_prob0}) in a decentralized manner. In other words, we need
to develop an optimization approach that enables players to make
decisions individually based on their local information while
coordinating to optimize the common objective $J^{u}_{\sigma}$. We
address this challenging problem by employing a new methodology
known as sensitivity-based optimization, and the details will be
introduced in the next section.

\section{Main Results}\label{section_result}
The main difficulty of the team variance optimization problem
(\ref{eq_prob0}) comes from the fact that the instantaneous variance
function defined in (\ref{eq_teamvar}) is not Markovian, since the
term $\mu^u$ in (\ref{eq_teamvar}) depends on the policy $u$ during
the whole process. Fortunately, we find that the team variance has
the following property which is useful to solve (\ref{eq_prob0}).

Without loss of generality, we consider $n$ random variables
$X_1,X_2,\dots,X_n$ and their set is denoted by
$X:=\{X_1,X_2,\dots,X_n\}$. The team variance of random variable set
$X$ is similarly defined as below.
\begin{equation}\label{eq_teamvarX}
\sigma := \sum_{i=1}^{n} \mathbbm E[(X_i - \mu)^2],
\end{equation}
where $\mu$ is the team mean of random variables $X_i$'s
\begin{equation}
\mu := \frac{1}{n}\sum_{i=1}^{n} \mathbbm E[X_i].
\end{equation}
We define \emph{pseudo team variance} of $X$ by replacing $\mu$ in
(\ref{eq_teamvarX}) with a constant $y \in \mathbb
R:=(-\infty,\infty)$, i.e.,
\begin{equation}\label{eq_pseudoteamvarX}
\sigma(y) := \sum_{i=1}^{n} \mathbbm E[(X_i - y)^2],
\end{equation}
where we call $y$ \emph{pseudo team mean}. We can derive the
following lemma about the property of pseudo team variance.
\begin{lemma}\label{lemma1}
The pseudo team variance $\sigma(y)$ is a quadratic convex function
of $y$ and its minimum is the true team variance $\sigma$ at $y^* =
\mu$.
\end{lemma}
\begin{proof}
We can directly rewrite (\ref{eq_pseudoteamvarX}) as follows.
\begin{eqnarray}
\sigma(y) = n y^2 - 2 y \sum_{i=1}^{n} \mathbbm E[X_i] +
\sum_{i=1}^{n} \mathbbm E[X_i^2].
\end{eqnarray}
Obviously, $\sigma(y)$ is quadratically convex in $y$. Its minimum
occurs at $y^* =  \frac{1}{n}\sum_{i=1}^{n} \mathbbm E[X_i] = \mu$
and the associated minimum value is $\sigma(y^*)=\sum_{i=1}^{n}
\mathbbm E[(X_i - \mu)^2]=\sigma$.
\end{proof}
Furthermore, we can derive that $\sigma$ and $\sigma(y)$ have the
following relation
\begin{equation}\label{eq11}
\sigma(y) = \sigma + n(y - \mu)^2.
\end{equation}
When these random variables are independent each other, similar to
(\ref{eq_var-relation}), we can further derive
\begin{equation}
\sigma[X] = \sum_{i=1}^{n}\sigma[X_i] + \sum_{i=1}^{n}(\mu_i -
\mu)^2,
\end{equation}
where $\sigma[X]$ is the team variance of random variable set $X$,
$\mu_i$ and $\sigma[X_i]$ are the mean and variance of random
variable $X_i$, respectively, with slight abuse of notation.

We apply the above results to our team variance optimization problem
(\ref{eq_prob0}). In (\ref{eq_teamvar}), we replace the team mean
$\mu^u$ with a constant $y$ and derive a pseudo team variance below
\begin{equation}\label{eq_pseudoteamvar}
J^u_{\sigma}(y) := \mathbbm E^u \left\{ \sum_{i=1}^{n}
\lim\limits_{T \rightarrow \infty}
\frac{1}{T}\sum_{t=0}^{T-1}[r_i(s_{i,t}, a_{i,t}) - y]^2 \right\},
\quad y \in \mathbb R.
\end{equation}
By applying (\ref{eq11}), we have
\begin{equation}\label{eq14}
J^u_{\sigma}(y) = J^u_{\sigma} + n(y - \mu^u)^2.
\end{equation}
Further with Lemma~\ref{lemma1}, we have
\begin{equation*}
J^u_{\sigma} = \min_{y \in \mathbb R} J^u_{\sigma}(y) =
J^u_{\sigma}(y^*)\Big|_{y^* = \mu^u}.
\end{equation*}
Thus, the optimization problem (\ref{eq_prob0}) is equivalent to the
following bilevel optimization problem
\begin{equation}\label{eq_bilevelprob}
J^*_{\sigma} = \min_{y \in \mathbb R}\min_{u \in \mathcal U^{\rm
HR}} J^u_{\sigma}(y),
\end{equation}
where $\min\limits_{y \in \mathbb R}$ and $\min\limits_{u \in
\mathcal U^{\rm HR}}$ are interchangeable. The optimization
objective is converted to the pseudo team variance
$J^u_{\sigma}(y)$. From the definition (\ref{eq_pseudoteamvar}), we
can see that $[r_i(s_{i,t}, a_{i,t}) - y]^2$ is independent of any
other players since $y$ is a constant. Therefore, with
Assumption~\ref{assm_1}, we can rewrite (\ref{eq_pseudoteamvar}) as
\begin{equation}\label{eq_pseudoteamvar2}
J^u_{\sigma}(y) = \sum_{i=1}^{n} \mathbbm E^{u_i} \left\{
\lim\limits_{T \rightarrow \infty}
\frac{1}{T}\sum_{t=0}^{T-1}[r_i(s_{i,t}, a_{i,t}) - y]^2 \right\} =
\sum_{i=1}^{n} J^{u_i}_{\sigma,i}(y),
\end{equation}
where $J^{u_i}_{\sigma,i}(y)$ is the \emph{pseudo variance} of
rewards of player $i$, which is defined similarly to
(\ref{eq_pseudoteamvar})
\begin{equation}\label{eq_pseudoteamvar-i}
J^{u_i}_{\sigma,i}(y) := \mathbbm E^{u_i} \left\{ \lim\limits_{T
\rightarrow \infty} \frac{1}{T}\sum_{t=0}^{T-1}[r_i(s_{i,t},
a_{i,t}) - y]^2 \right\}, \quad y \in \mathbb R.
\end{equation}

By summarizing the above results (\ref{eq14}),
(\ref{eq_pseudoteamvar2}), and (\ref{eq_var-relation}), we directly
derive the following lemma.
\begin{lemma}\label{lemma2}
There are the following relations between the team variance and
pseudo team variance in stochastic games
\begin{equation}\label{eq18}
J^u_{\sigma} = J^u_{\sigma}(y) - n(y - \mu^u)^2, \qquad u \in
\mathcal U^{\rm{HR}}, y \in \mathbb R.
\end{equation}
With Assumption~\ref{assm_1} of separately controlled chains, we
have
\begin{equation}\label{eq19}
J^u_{\sigma} = \sum_{i=1}^n\left[J^{u_i}_{\sigma,i}(y) - (y -
\mu^u)^2 \right] = \sum_{i=1}^n\left[J^{u_i}_{\sigma,i} +
(\mu^{u_i}_i - \mu^u)^2 \right], \quad u \in \mathcal U^{\rm{HR}}, y
\in \mathbb R.
\end{equation}
\end{lemma}

With Assumption~\ref{assm_1}, the inner level optimization problem
in (\ref{eq_bilevelprob}) can be decoupled as the sum of $n$
subproblems
\begin{flalign}\label{eq_innerprob}
&\mathcal M(y): \hspace{3.5cm}  \min_{u \in \mathcal U^{\rm HR}}
J^u_{\sigma}(y) = \sum_{i=1}^{n}\min_{u_i \in \mathcal U^{\rm HR}_i}
J^{u_i}_{\sigma,i}(y), \qquad y \in \mathbb R, &
\end{flalign}
where we rewrite the subproblem of optimizing the pseudo variance of
a single individual player $i$ as
\begin{flalign}\label{eq_subinnerprob}
&\mathcal M_i(y): \hspace{4.5cm}   \min_{u_i \in \mathcal U^{\rm
HR}_i} J^{u_i}_{\sigma,i}(y),  \qquad y \in \mathbb R, i \in
\mathcal I. &
\end{flalign}
Therefore, we can directly derive the following theorem.

\begin{theorem}\label{theorem1}
With Assumption~\ref{assm_1}, the team variance optimization problem
$\mathcal M$ in (\ref{eq_prob0}) is equivalent to a bilevel
optimization problem whose inner level problem is decoupled into a
series of pseudo variance optimization subproblems $\mathcal M_i(y)$
in (\ref{eq_subinnerprob}), i.e.,
\begin{equation}\label{eq_probeq}
\min_{u \in \mathcal U^{\rm HR}} J^{u}_{\sigma} \Longleftrightarrow
\min_{y \in \mathbb R} \sum_{i=1}^{n}\min_{u_i \in \mathcal U^{\rm
HR}_i} J^{u_i}_{\sigma,i}(y).
\end{equation}
\end{theorem}

When $y$ is fixed as a constant, we observe that the instantaneous
cost function $[r_i(s_{i,t},a_{i,t})-y]^2$ in
(\ref{eq_pseudoteamvar}) or (\ref{eq_pseudoteamvar-i}) is Markovian.
The inner optimization problem $\mathcal M(y)$ or $\mathcal M_i(y)$,
i.e., $\min\limits_{u \in \mathcal U^{\rm HR}} J^u_{\sigma}(y)$ or
$\min\limits_{u_i \in \mathcal U^{\rm HR}_i} J^{u_i}_{\sigma,i}(y)$,
is a standard MDP with long-run average criterion. Since it is well
known that a standard MDP can attain optimum in the space of
stationary deterministic policies \citep{Puterman94}, with
Lemma~\ref{lemma2} we have the following theorem.

\begin{theorem}\label{theorem2}
There always exists a stationary deterministic policy that can
achieve the minimal team variance $J^*_{\sigma}$ in
(\ref{eq_prob0}), which is also a stationary pure Nash equilibrium
policy in games.
\end{theorem}

With Theorem~\ref{theorem2}, we can focus on the stationary
deterministic policy space. In the rest of the paper, all the
previous $\mathcal U^{\rm HR}$ and $\mathcal U^{\rm HR}_i$ will be
replaced with $\mathcal U^{\rm SD}$ and $\mathcal U^{\rm SD}_i$,
respectively. Thus, the bilevel optimization problem in
(\ref{eq_probeq}) can be rewritten as
\begin{equation}\label{eq_bilevelprob2}
\min_{u \in \mathcal U^{\rm HR}} J^{u}_{\sigma} \Longleftrightarrow
\min_{y \in \mathbb R} \sum_{i=1}^{n}\min_{u_i \in \mathcal U^{\rm
SD}_i} J^{u_i}_{\sigma,i}(y).
\end{equation}
To solve the above bilevel optimization problem, we need to
enumerate each possible $y$ and solve the inner MDP problems, i.e.,
we have to solve the MDP problems $\mathcal M_i(y)$ for all $i \in
\mathcal I$ and $y \in \mathbb R$, which is computationally
intractable. Therefore, it is desirable to develop more efficient
approaches to solve this problem. Below, we use the
sensitivity-based optimization theory \citep{Cao07} to study this
problem.

Performance difference formula is a key result of the
sensitivity-based optimization theory and it quantifies the
performance change of a Markov system if its policy has
perturbations. We consider a player $i \in \mathcal I$ with policy
$u_i \in \mathcal U^{\rm SD}_i$. If its policy is perturbed to a new
one $u'_i \in \mathcal U^{\rm SD}_i$, the change of the pseudo
variance under these two policies is quantified by the following
performance difference formula \citep{Cao07}
\begin{equation}\label{eq_diffi}
J^{u'_{i}}_{\sigma,i}(y) - J^{u_{i}}_{\sigma,i}(y) = \bm
\pi^{u'_i}_i[(\bm P^{u'_i}_i - \bm P^{u_i}_i)\bm g^{u_i}_i + (\bm
r^{u'_i}_i - y \bm 1)^2_{\odot} - (\bm r^{u_i}_i - y \bm
1)^2_{\odot}],
\end{equation}
where $\bm \pi^{u'_i}_i$ is a row vector representing the
steady-state distribution of player $i$ under policy $u'_i$, $\bm
P^{u'_i}_i$ and $\bm P^{u_i}_i$ are state transition probability
matrices of player $i$ under policies $u'_i$ and $u_i$,
respectively, $\bm r^{u'_i}_i$ and $\bm r^{u_i}_i$ are column
vectors representing the instantaneous cost function of player $i$
under policies $u'_i$ and $u_i$, respectively, i.e., their elements
are given by $r_i(s_i,u'_i(s_i))$ and $r_i(s_i,u_i(s_i))$,
respectively, $\bm 1$ is a column vector with elements $1$,
$(\cdot)^2_{\odot}$ represents the elementwise square of a vector,
and $\bm g^{u_i}_i$ is a column vector called the performance
potential of player $i$ under policy $u_i$. An element of $\bm
g^{u_i}_i$ is defined as below.
\begin{equation}\label{eq_bias}
g_i^{u_i}(s) := \mathbbm E \left\{ \sum_{t=0}^{\infty}
[r_i(s_{i,t},a_{i,t}) - y]^2 - J^{u_{i}}_{\sigma,i}(y)
\Big|_{s_{i,0}=s} \right\}, \quad s \in \mathcal S_i,
\end{equation}
where $a_{i,t} = u_i(s_{i,t})$ and $J^{u_{i}}_{\sigma,i}(y)$ is the
long-run average performance of player $i$ under policy $u_i$ which
is defined in (\ref{eq_pseudoteamvar-i}). Therefore, $\bm g^{u_i}_i$
quantifies the expectation of accumulated bias between every
instantaneous cost and long-run average cost, which is also called
relative value function or bias in the theory of long-run average
MDPs \citep{Puterman94}.

Furthermore, we consider the difference of team variance in
stochastic games when each player has perturbations on their own
policies, i.e., the current policies $u_i$ are changed to new
policies $u'_i$, where $u_i, u'_i \in \mathcal U^{\rm SD}_i$, $i \in
\mathcal I$. By applying the first equality in (\ref{eq19}) to
(\ref{eq_diffi}), we derive
\begin{eqnarray}
J^{u'}_{\sigma} - J^{u}_{\sigma} &=& \sum_{i=1}^{n}
[J^{u'_{i}}_{\sigma,i}(y) - J^{u_{i}}_{\sigma,i}(y) - (y-\mu^{u'})^2 + (y-\mu^{u})^2] \nonumber\\
&=& \sum_{i=1}^{n} \left\{ \bm \pi^{u'_i}_i[(\bm P^{u'_i}_i - \bm
P^{u_i}_i)\bm g^{u_i}_i + (\bm r^{u'_i}_i - y \bm 1)^2_{\odot} -
(\bm r^{u_i}_i - y \bm 1)^2_{\odot}] - (y-\mu^{u'})^2 + (y-\mu^u)^2
\right\}. \nonumber
\end{eqnarray}
By letting $y = \mu^u$, we derive the \emph{team variance difference
formula} in stochastic games under any two policies $u$ and $u'$,
where $u, u' \in \mathcal U^{\rm SD}$ and Assumption~\ref{assm_1}
holds.
\begin{center}
\begin{boxedminipage}{1\columnwidth}
\vspace{-10pt}
\begin{equation}\label{eq_diffvar}
J^{u'}_{\sigma} - J^{u}_{\sigma} = \sum_{i=1}^{n} \bm
\pi^{u'_i}_i[(\bm P^{u'_i}_i - \bm P^{u_i}_i)\bm g^{u_i}_i + (\bm
r^{u'_i}_i - \mu^u \bm 1)^2_{\odot} - (\bm r^{u_i}_i - \mu^u \bm
1)^2_{\odot}] - n(\mu^{u'}-\mu^{u})^2.
\end{equation}
\vspace{-15pt}
\end{boxedminipage}
\end{center}

The above difference formula plays a fundamental role in the team
variance optimization of stochastic games. The first part of the
right-hand-side of (\ref{eq_diffvar}) can be viewed as a
decentralized optimization problem of each player $i$, which can be
solved by standard MDP approaches. The second part
$n(\mu^{u'}-\mu^{u})^2$ is always non-negative, although we do not
know the exact value of $\mu^{u'}$ before the new policy $u'$ is
determined. Moreover, solving the optimization problem of the first
part only needs the local information of player $i$, i.e., the
values of transition probability matrices $\bm P^{u'_i}_i$ and $\bm
P^{u_i}_i$, reward functions $\bm r^{u'_i}_i$ and $\bm r^{u_i}_i$,
and the relative value function $\bm g^{u_i}_i$, which are either
given or computable parameters. The value of the team mean $\mu^u$
under the current policy $u$ is also computable or estimatable,
which can be viewed as a coordination signal. Solving an inner MDP
problem $\mathcal M_i(y)$ can ensure the difference value in
(\ref{eq_diffi}) is negative when $y=\mu^u$, thus the team variance
difference value in (\ref{eq_diffvar}) is guaranteed negative and
the team variance is reduced under the new policy $u'$. This is the
main idea of optimizing the team variance in stochastic games.

\noindent \textbf{Remark 1.} With the team variance difference
formula (\ref{eq_diffi}), we only need to solve a series of inner
MDPs $\mathcal M_i(y)$ where $y = \mu^u$ and $i \in \mathcal I$. Any
new policy $u'_i$ that improves the performance of $\mathcal M_i(y)$
can reduce the team variance of the whole stochastic game. The
optimization of each $\mathcal M_i(y)$ can be conducted separately
only based on their local information, which makes the optimization
of team variance of stochastic games solvable in a decentralized
manner. The team mean $\mu^u$ serves as a signal to coordinate the
optimization direction of each player.

With the above observations, we directly derive the following
theorem that guides the optimization of team variance
$J^u_{\sigma}$.

\begin{theorem}\label{theorem3}
If we find a new policy $u'_i \in \mathcal U^{\rm SD}_i$ that
ensures the vector $(\bm P^{u'_i}_i - \bm P^{u_i}_i)\bm g^{u_i}_i +
(\bm r^{u'_i}_i - \mu^u \bm 1)^2_{\odot} - (\bm r^{u_i}_i - \mu^u
\bm 1)^2_{\odot} \leq \bm 0$ elementwisely, then we have
$J^{u'}_{\sigma} \leq J^{u}_{\sigma}$. If the inequality holds
strictly in at least one element, then $J^{u'}_{\sigma} <
J^{u}_{\sigma}$.
\end{theorem}

The proof of Theorem~\ref{theorem3} is straightforward based on the
team variance difference formula (\ref{eq_diffvar}), by utilizing
the fact that each element of the steady-state distribution vector
$\bm \pi^{u'_i}_i$ is always positive in any ergodic Markov chains.
Based on (\ref{eq_diffvar}), we can further derive a \emph{necessary
condition} of optimal policy $u^*$ for the team variance
minimization problem (\ref{eq_prob0}).

\begin{theorem}\label{theorem4}
The optimal policy $u^*=(u^*_1,u^*_2,\dots,u^*_n)$ of the team
variance minimization problem (\ref{eq_prob0}) must satisfy $\bm
P^{u'_i}_i \bm g^{u^*_i}_i + (\bm r^{u'_i}_i - \mu^{u^*} \bm
1)^2_{\odot} \geq \bm P^{u^*_i}_i \bm g^{u^*_i}_i + (\bm r^{u^*_i}_i
- \mu^{u^*} \bm 1)^2_{\odot}$ elementwisely, for any $u'_i \in
\mathcal U^{\rm SD}_i$ and $i \in \mathcal I$.
\end{theorem}

Theorem~\ref{theorem4} can be proved by using the contradiction
method based on the team variance difference formula
(\ref{eq_diffvar}). For simplicity, we omit the detailed proof.

Based on the equivalence of bilevel optimization problem
(\ref{eq_bilevelprob2}) and Theorem~\ref{theorem3}, we can develop
an iterative procedure to repeatedly reduce the team variance of
stochastic games. For the current policy $u$, we can estimate its
team mean $\mu^u$. Letting $y=\mu^u$, we solve a series of
subproblems $\mathcal M_i(y)$ and derive a set of new policies $u'$
which ensures $J^{u'}_{\sigma} < J^{u}_{\sigma}$ by using
Theorem~\ref{theorem3}. Furthermore, we update the team mean under
the new policy $u'$ and let $y=\mu^{u'}$, repeat the above procedure
and obtain a new policy $u''$ which ensures $J^{u''}_{\sigma} <
J^{u'}_{\sigma}$. Thus, the team variance is repeatedly reduced and
approaches the optimum $J^*_{\sigma}$. The above optimization
procedure is formally stated by the following algorithm.

\begin{center}
\begin{boxedminipage}{1\columnwidth}
\vspace{5pt}

\noindent\textbf{ Initialization}
\begin{itemize}
\item{For each player $i \in \mathcal I$, arbitrarily choose an initial policy $u^{(0)}_i$ from the policy space $\mathcal U^{\rm SD}_i$, and set $l=0$.}
\end{itemize}

\noindent\textbf{ Policy Evaluation}
\begin{itemize}
\item{For the current policy $u^{(l)}=(u^{(l)}_1,u^{(l)}_2,\dots,u^{(l)}_n)$, numerically compute or estimate $\mu^{u^{(l)}}$ and $\bm g^{u_i^{(l)}}_i$
based on their definitions (\ref{eq_teammean}) and (\ref{eq_bias}),
respectively, $i \in \mathcal I$.}
\end{itemize}

\noindent\textbf{ Policy Improvement}
\begin{itemize}
\item{For each player $i \in \mathcal I$, separately update its policy based on
its local information and the coordination signal $\mu^{u^{(l)}}$ as
follows.
\begin{equation}\label{eq_policyimprov}
u^{(l+1)}_i(s) = \argmin\limits_{a \in \mathcal A_i}\left\{
(r_i(s,a) - \mu^{u^{(l)}})^2 + \sum_{s' \in \mathcal
S_i}p_i(s'|s,a)g^{u_i^{(l)}}_{i}(s') \right\},  \quad s \in \mathcal
S_i.
\end{equation}
When there exist multiple actions attaining the above minimum
operation, choose $u_i^{(l+1)}(s)=u_i^{(l)}(s)$ if possible to avoid
oscillations.}
\end{itemize}

\noindent\textbf{ Stopping Rule}
\begin{itemize}
\item{If $u^{(l+1)} = u^{(l)}$, stop; Otherwise, let $l \leftarrow l+1$ and go to step 2.}
\end{itemize}

\vspace{3pt}
\end{boxedminipage}

\textbf{Algorithm 1.} A decentralized algorithm to reduce the team
variance of stochastic games.
\end{center}

Based on the previous analysis results, we can prove that
Algorithm~1 will converge to a strictly local minimum or a
first-order stationary point in the mixed policy space, which is
stated by the following theorem.

\begin{theorem}
Algorithm~1 will converge to a strictly local minimum or a
first-order stationary point in the space of mixed policies after a
finite number of iterations.
\end{theorem}
\begin{proof}
First, we prove the convergence of Algorithm~1 within a finite
number of iterations. From (\ref{eq_policyimprov}), we can see that
$u_i^{(l+1)}$ and $u_i^{(l)}$ must satisfy the strict inequality in
Theorem~\ref{theorem3} since $u_i^{(l+1)} \neq u_i^{(l)}$ holds for
at least one player $i$. Therefore, we can derive that
$J^{u^{(l+1)}}_{\sigma} < J^{u^{(l)}}_{\sigma}$ and the team
variance is strictly reduced after each iteration of Algorithm~1.
Since the policy space $\mathcal U^{\rm SD}$ is finite, Algorithm~1
will surely stop after a finite number of iterations.

Second, we study the property of the convergence point. We define a
mixed policy $\delta^{u'}_{u}$ that perturbs from the current policy
$u$ to a new policy $u'$ with a mixing probability $\delta$, where
$u,u' \in \mathcal U^{\rm SD}$ and $0\leq \delta \leq 1$. In other
words, if the system adopts the mixed policy $\delta^{u'}_{u}$, each
player~$i \in \mathcal I$ will adopt policy $u_i$ with probability
$1-\delta$ and adopt policy $u'_i$ with probability $\delta$.
Replacing $\delta^{u'}_{u}$ with $u'$ in (\ref{eq_diffvar}) and
taking the differentiating operation through $\delta \rightarrow 0$,
we can derive the \emph{team variance derivative formula} as
follows.
\begin{equation}\label{eq_derivvar}
\frac{\partial J^{\delta^{u'}_{u}}_{\sigma}}{\partial \delta}
\Big|_{\delta=0} = \sum_{i=1}^{n} \bm \pi^{u_i}_i[(\bm P^{u'_i}_i -
\bm P^{u_i}_i)\bm g^{u_i}_i + (\bm r^{u'_i}_i - \mu^u \bm
1)^2_{\odot} - (\bm r^{u_i}_i - \mu^u \bm 1)^2_{\odot}].
\end{equation}
The above derivative can also be denoted as $\nabla J^{u}_{\sigma}$,
which is the gradient of the team variance $J^{u}_{\sigma}$ with
respect to the mixing probability $\delta$ along with a mixing
direction $u \rightarrow u'$ in the policy space $\mathcal U^{\rm
SD}$. When Algorithm~1 stops at a policy, say $u^*$, it indicates
that (\ref{eq_policyimprov}) cannot be improved anymore, i.e.,
\begin{equation}\label{eq29}
\bm P^{u^*_i}_i \bm g^{u^*_i}_i + (\bm r^{u^*_i}_i - \mu^{u^*} \bm
1)^2_{\odot} \leq \bm P^{u'_i}_i \bm g^{u^*_i}_i + (\bm r^{u'_i}_i -
\mu^{u^*} \bm 1)^2_{\odot}, \quad \forall u'_i \in \mathcal U^{\rm
SD}_i, i \in \mathcal I.
\end{equation}
This indicates that the converged point $u^*$ satisfies the
necessary condition of optimal policies, as stated in
Theorem~\ref{theorem4}. Replacing $u$ with $u^*$ in
(\ref{eq_derivvar}) and substituting (\ref{eq29}) into
(\ref{eq_derivvar}), we derive
\begin{equation}\label{eq30}
\frac{\partial J^{\delta^{u'}_{u^*}}_{\sigma}}{\partial \delta}
\Big|_{\delta=0} = \sum_{i=1}^{n} \bm \pi^{u^*_i}_i[(\bm P^{u'_i}_i
- \bm P^{u^*_i}_i)\bm g^{u^*_i}_i + (\bm r^{u'_i}_i - \mu^{u^*} \bm
1)^2_{\odot} - (\bm r^{u^*_i}_i - \mu^{u^*} \bm 1)^2_{\odot}] \geq
0, \quad \forall u' \in \mathcal U^{\rm SD}.
\end{equation}
For more rigorously discussing the sign of the above derivative, we
further partition the policy space $\mathcal U^{\rm SD}$ into two
parts: $\tilde{\mathcal U}^{\rm SD}_{u^*}$ and $\mathcal U^{\rm
SD}\backslash\tilde{\mathcal U}^{\rm SD}_{u^*}$, where
$\tilde{\mathcal U}^{\rm SD}_{u^*}$ is defined as
\begin{equation}\label{eq31}
\tilde{\mathcal U}^{\rm SD}_{u^*} := \{u' :  \frac{\partial
J^{\delta^{u'}_{u^*}}_{\sigma}}{\partial \delta} \Big|_{\delta=0} =
0\}.
\end{equation}
Thus, with (\ref{eq30}) we have
\begin{equation}
\mathcal U^{\rm SD}\backslash\tilde{\mathcal U}^{\rm SD}_{u^*} :=
\{u' :  \frac{\partial J^{\delta^{u'}_{u^*}}_{\sigma}}{\partial
\delta} \Big|_{\delta=0} > 0\}.
\end{equation}
Therefore, we can see that the converged point $u^*$ is a
first-order stationary point in the mixed policy space spanned by
$u' \in \tilde{\mathcal U}^{\rm SD}_{u^*}$ and a strictly local
minimum in the mixed policy space spanned by $u' \in \mathcal U^{\rm
SD}\backslash\tilde{\mathcal U}^{\rm SD}_{u^*}$. Thus, the theorem
is proved.
\end{proof}

We can see that Algorithm~1 is a decentralized optimization
algorithm and it can be implemented by each player independently
based on their own local information. The decentralized optimization
procedure is coordinated by a signal of team mean $\mu^u$. The
algorithm will strictly reduce the whole team variance
$J^u_{\sigma}$ after every player's iteration. The converged policy
is usually a strictly local minimum in the mixed policy space since
the condition in (\ref{eq31}) is hardly satisfied and the set
$\tilde{\mathcal U}^{\rm SD}_{u^*}$ is usually null in most
scenarios.

\section{Numerical Experiment}\label{section_experiment}
In this section, we use the energy management problem in a multiple
microgrid system to demonstrate the effectiveness of our approach,
where Assumptions~\ref{assm_1} and \ref{assm_2} naturally hold and
the team variance is a suitable metric to measure the stability and
fairness of microgrids.

\begin{figure}[hbtp]
\centering
\includegraphics[width=.75\textwidth]{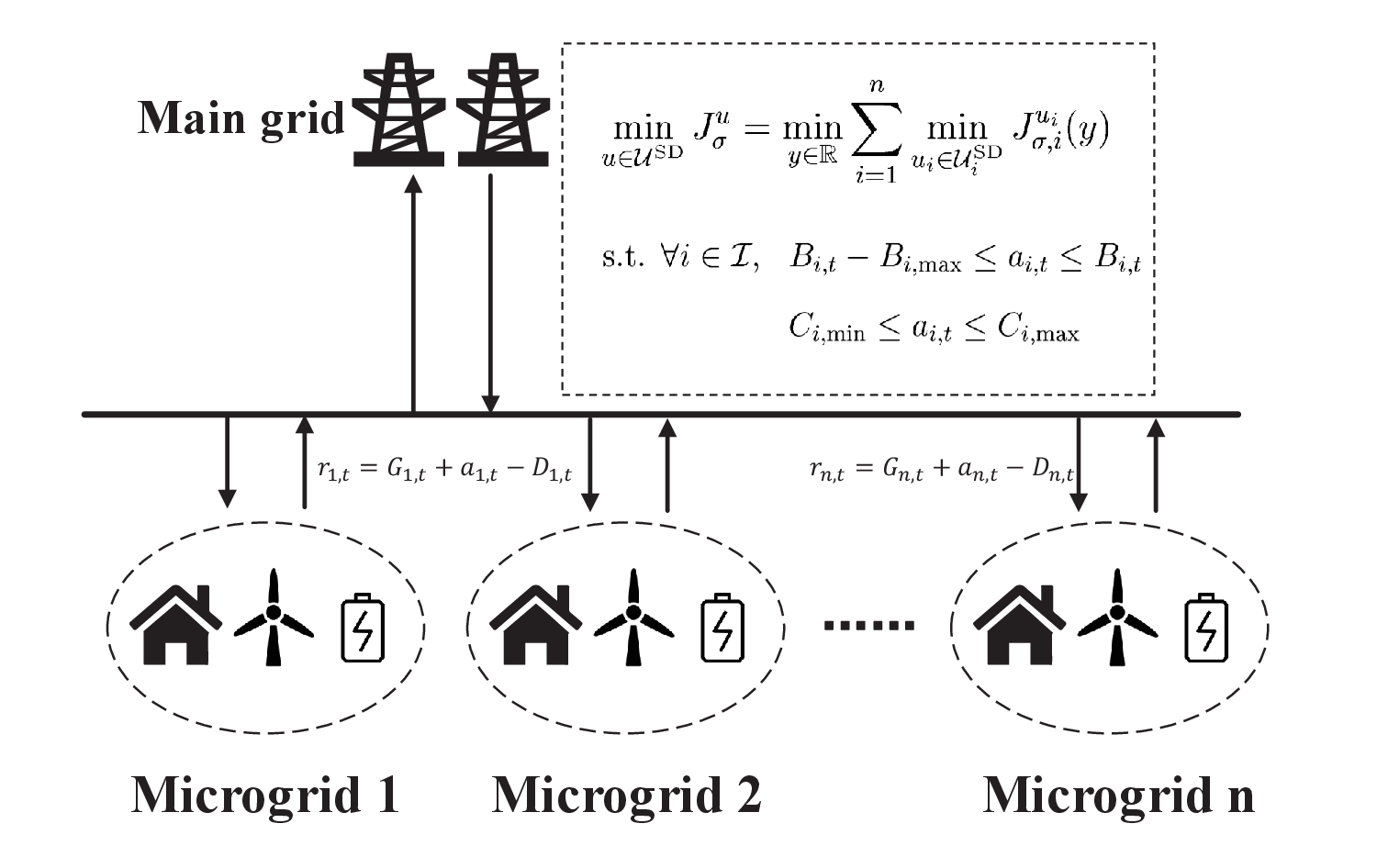}
\caption{Illustration of the team variance optimization in a
multiple microgrid system.} \label{fig_microgrid}
\end{figure}

As illustrated in Figure~\ref{fig_microgrid}, we consider an energy
management system including $n$ microgrids connected with the main
grid. Each microgrid $i \in \mathcal I$ is equipped with wind
turbines to generate electricity $G_{i,t}$ and batteries to store
energy $a_{i,t}$. The decision variable $a_{i,t}$ is defined as
follows: when $a_{i,t} > 0$, it represents the discharging power of
the batteries during periods of low wind power availability; when
$a_{i,t} < 0$, it represents the charging power of the batteries
when there is an abundance of wind power. The decision variable
$a_{i,t}$ has to satisfy the constraints: $C_{i,\min} \leq a_{i,t}
\leq C_{i,\max}$ where $C_{i,\max}$ and $-C_{i,\min}$ represent the
maximum discharging and charging power of the batteries,
respectively; $B_{i,t} - B_{i,\max} \leq a_{i,t} \leq B_{i,t}$ where
$B_{i,t}$ is the energy level of the batteries at time $t$ and
$B_{i,\max}$ is the maximum capacity of the batteries. The state of
microgrid $i$ is defined as $s_{i,t}:=(G_{i,t}, B_{i,t})$. The
dynamic of $G_{i,t}$ is a stochastic process determined by the wind
power distribution, which is modeled by a Markov chain with
probability matrix $\bm P_g$ calculated by statistics. The dynamic
of $B_{i,t}$ is determined by action $a_{i,t}$ as
$B_{i,t+1}=B_{i,t}-a_{i,t}$. Each microgrid $i$ has its own
electricity demand $D_{i,t}$, which is known prior by demand
prediction. The exchanged power between microgrid $i$ and the main
grid is denoted by $r_{i,t}=G_{i,t}+a_{i,t}-D_{i,t}$: $r_{i,t}>0$
represents selling energy to the main grid; $r_{i,t}<0$ represents
buying energy from the main grid. The control policy
$u=(u_1,u_2,\dots,u_n)$ is stationary and deterministic and its
element $u_i(s_{i})$ determines the discharging power $a_{i}$, $i
\in \mathcal I$. As we know, the power supply and demand of grid
must be equal at every time. Thus, the fluctuation of the power
exchanged between the main grid and microgrids affects the stability
and safety of the power system, which can be measured by the
variance of $r_{i,t}$. On the other hand, we are also concerned
about the fairness among these $n$ microgrids, which indicates that
the deviation of the mean values of $r_{i,t}$'s of $n$ microgrids
should be minimized. Thus, we aim to minimize the team variance
$J^{u}_{\sigma}$ which reflects our concerns on both stability and
fairness of multiple microgrids.

We verify Assumptions~\ref{assm_1} and \ref{assm_2} in this
application example. In the state variable
$s_{i,t}=(G_{i,t},B_{i,t})$, the wind power $G_{i,t}$ is determined
by weather which is not affected by the discharging power
$a_{i,t}$'s of microgrids, the battery energy level $B_{i,t}$ is
determined only by its own action $a_{i,t}$, not by any other
microgrids' actions. Therefore, we can see that $p_i(s'_i|s_i, \bm
a) = p_i(s'_i|s_i, a_i)$ and $r_i(\bm s,\bm a)=r_i(s_i,a_i)$, and
Assumption~\ref{assm_1} of separately controlled chains holds
naturally in this example. It is reasonable that each microgrid
knows only the local information about its own wind power, battery
energy level, and discharging power. Other microgrids have no
willingness to share their own information because of privacy. Thus,
Assumption~\ref{assm_2} about local observations is also reasonable
in this example. Therefore, we can apply our main results to study
this energy management problem and minimize the team variance by
using Algorithm~1 in a decentralized manner.


The parameter settings of this experiment are listed in
Tables~\ref{table1} and \ref{table2} as below.

\begin{table}[htbp]
    \centering
    \caption{Parameter settings of microgrids.}\label{table1}
    \begin{tabular}{ccc}
        \hline
        \toprule
        \# of microgrids $n$ & capacity of batteries $B_{i,\max}$/MWh \\ \hline
        3 & 5 \\ \hline\hline
        maximum charging power $C_{i,\min}$/MW & maximum discharging power $C_{i,\max}$/MW \\ \hline
        $-2$ & $+2$ \\ \hline\hline
    \end{tabular}
\end{table}

\begin{table}[htbp]
    \centering
    \caption{Discretized states and actions of microgrids.}\label{table2}
    \begin{tabular}{ccccccc}
        \hline
        \toprule
        State $G_i$ (wind power/MW) & 0 & 1 & 2 & 3 & 4 & 5 \\ \hline
        State $B_i$ (storage energy level/MWh) & 0 & 1 & 2 & 3 & 4 & 5 \\ \hline
        Action $a_i$ (battery (dis)charing power/MW) & $-2$ & $-1$ & 0 & $1$ & $2$  \\ \hline
    \end{tabular}
\end{table}


We assume that the wind power of these microgrids is independently
distributed, which is reasonable if they are located far away. The
transition probability matrix $\bm P_{g,1}$ of wind power states of
microgrid~1 is approximately calculated based on the real data from
the Measurement and Instrumentation Data Center (MIDC) in the
National Renewable Energy Laboratory (NREL, online), which maintains
a data set of wind speed at an experiment field since 1996. The
transition probability matrices $\bm P_{g,2}$ and $\bm P_{g,3}$ of
microgrid~2 and 3 are slightly manipulated based on $\bm P_{g,1}$,
which reflect a relatively strong and weak wind profile,
respectively. For simplicity, we assume that the electricity demand
of each microgrid is a constant, and we set $D_{1,t} = 2, D_{2,t} =
2.5, D_{3,t} = 2$ for the diversity of different microgrids. In this
example, we allow wind abandonment to ensure that the power selling
of any microgrid to the main grid is not larger than 2.

\begin{equation*}
\bm{P}_{g,1}= \left(
\begin{matrix}
0.53 & 0.18 & 0.19 & 0.04 & 0.01 & 0.05\\
0.51 & 0.08 & 0.20 & 0.08 & 0.02 & 0.11\\
0.35 & 0.11 & 0.19 & 0.11 & 0.03 & 0.21\\
0.27 & 0.15 & 0.15 & 0.14 & 0.03 & 0.26\\
0.14 & 0.11 & 0.13 & 0.15 & 0.05 & 0.42\\
0.09 & 0.03 & 0.06 & 0.06 & 0.03 & 0.73
\end{matrix}
\right),
\end{equation*}
\begin{equation*}\bm{P}_{g,2}= \left(
\begin{matrix}
0.33 & 0.18 & 0.19 & 0.04 & 0.01 & 0.25\\
0.31 & 0.08 & 0.20 & 0.08 & 0.02 & 0.31\\
0.15 & 0.11 & 0.19 & 0.11 & 0.03 & 0.41\\
0.17 & 0.15 & 0.15 & 0.14 & 0.03 & 0.36\\
0.04 & 0.11 & 0.13 & 0.15 & 0.05 & 0.52\\
0.07 & 0.03 & 0.06 & 0.06 & 0.03 & 0.75\\
\end{matrix}
\right), \quad \bm{P}_{g,3}= \left(
\begin{matrix}
0.53 & 0.18 & 0.19 & 0.04 & 0.01 & 0.05\\
0.51 & 0.08 & 0.20 & 0.08 & 0.02 & 0.11\\
0.45 & 0.11 & 0.19 & 0.11 & 0.03 & 0.11\\
0.37 & 0.15 & 0.15 & 0.14 & 0.03 & 0.16\\
0.34 & 0.11 & 0.13 & 0.15 & 0.05 & 0.22\\
0.49 & 0.03 & 0.06 & 0.06 & 0.03 & 0.33\\
\end{matrix}
\right).
\end{equation*}

We apply Algorithm~1 to this example of energy management. The
initial policy is randomly generated and each microgrid optimizes
their own pseudo variance based on their own local information and a
coordinated signal of the current team mean $\mu^{u^{(l)}}$. The
optimization procedure of Algorithm~1 is illustrated by
Figure~\ref{fig_2}, where the results of microgrid~2 are given as
representation. Figure~\ref{fig_2} indicates that Algorithm~1 can
converge after 6 iterations and the team variance is optimized from
$J^{u^{(0)}}_\sigma=10.1243$ to $J^{u^{(6)}}_\sigma=4.3440$. From
Figure~\ref{fig_2}(a), we can see that the team variance is
monotonically reduced and the team mean has small fluctuations
during the optimization procedure. From Figure~\ref{fig_2}(b), we
can see that the pseudo variance of microgrid~2 is monotonically
reduced, under the coordination of signal $\mu^u$. However, the
variance of microgrid~2 is not monotonically reduced: the variance
at the 3rd iteration is larger than that at the 2nd iteration. From
this experiment, we validate that Algorithm~1 can effectively reduce
the team variance of stochastic games in a decentralized manner.
Moreover, the optimization usually converges fast and has large
improvements during the first few iterations, which resemble the
performance behavior of classic policy iteration algorithms in MDPs.

\begin{figure}[htbp]
    \subfigure[Convergence of team variance and mean.]{
        \begin{minipage}[b]{0.5\textwidth}
            \includegraphics[width=1\textwidth]{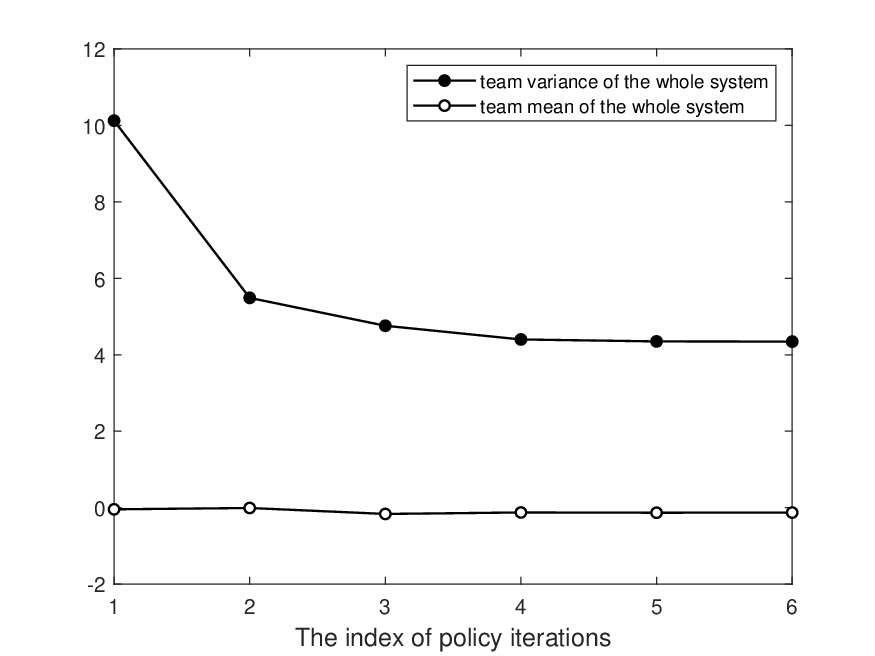}
        \end{minipage}
    }
    \subfigure[Convergence of values in microgrid~2.]{
        \begin{minipage}[b]{0.5\textwidth}
            \includegraphics[width=1\textwidth]{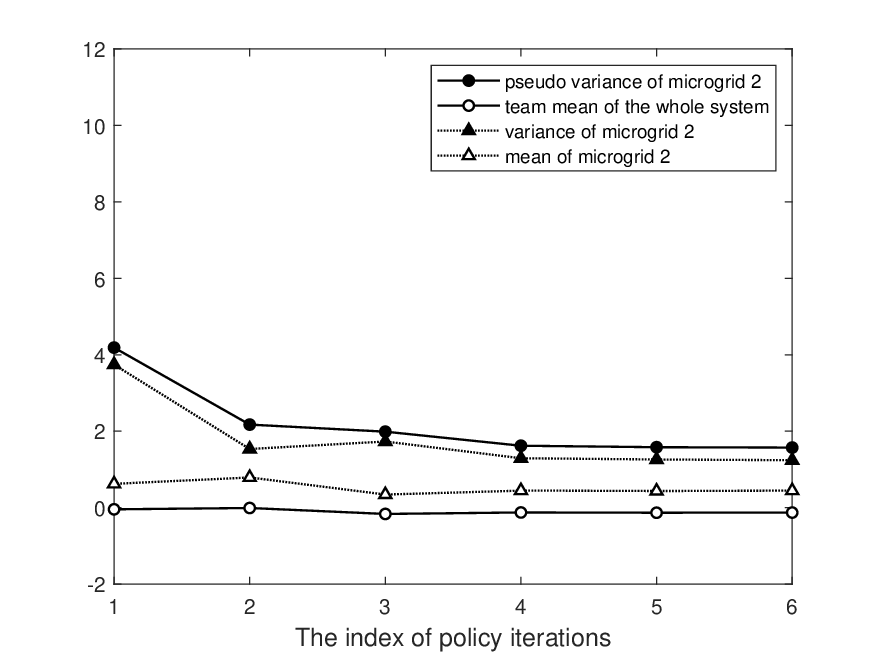}
        \end{minipage}
    }
    \caption{The policy iteration procedures of Algorithm~1. } \label{fig_2}
\end{figure}

\section{Conclusion} \label{section_conclusion}
In this paper, we investigate team variance optimization in
$n$-player stochastic games with separately controlled chains. By
leveraging sensitivity-based optimization, we address the challenges
posed by the non-additivity and non-Markovian nature of the variance
metric, as well as the constraint of local information
observability. We derive two key sensitivity optimization formulas:
the difference formula and the derivative formula for team variance,
which provide a solid theoretical foundation for decentralized
variance optimization. Furthermore, we establish the existence of a
stationary pure Nash equilibrium policy. Building on these results,
we propose a bilevel optimization algorithm that effectively
optimizes team variance in a decentralized manner, where the team
mean at the outer level serves as a coordinating signal for the
inner-level optimization of the $n$ players. We prove that the
proposed algorithm converges to either a strictly local minimum or a
first-order stationary point in the space of mixed policies,
demonstrating its theoretical soundness and practical utility.

It is a promising future direction to explore team variance-aware
multi-agent reinforcement learning in a data-driven manner. With the
abundance of data in real-world applications, such an approach could
significantly enhance the feasibility and scalability of
optimization algorithms. Leveraging deep learning techniques to
approximate value functions and agent policies could be particularly
important in this context. Additionally, extending the team variance
optimization framework to other domains, such as fairness
optimization in queueing systems or transportation networks, offers
interesting opportunities for future work. These scenarios,
characterized by the presence of multiple decision-makers, highlight
the broad applicability and relevance of our proposed framework.

\section*{Funding Declaration}
This work was supported in part by the National Natural Science
Foundation of China (72342006).


\end{document}